\documentclass{amsart}

\usepackage{amssymb,amsmath,latexsym,array,delarray,amsthm,amssymb,epsfig,setspace,multicol,titletoc,tocloft,graphicx,mathtools,soul,xcolor,enumitem,csquotes,indentfirst,tikz}

\usepackage[indentafter]{titlesec}
\titleformat{name=\section}{}{\thetitle.}{0.8em}{\centering\scshape}
\titleformat{name=\subsection}[runin]{}{\thetitle.}{0.5em}{\bfseries}[.]
\titleformat{name=\subsubsection}[runin]{}{\thetitle.}{0.5em}{\itshape}[.]
\titleformat{name=\paragraph,numberless}[runin]{}{}{0em}{}[.]
\titlespacing{\paragraph}{0em}{0em}{0.5em}
\titleformat{name=\subparagraph,numberless}[runin]{}{}{0em}{}[.]
\titlespacing{\subparagraph}{0em}{0em}{0.5em}

\usepackage{array,multirow}

\newcolumntype{L}[1]{>{\raggedright\arraybackslash}m{#1}}

\usepackage[top=1 in, bottom=1in, left=1.5 in, right=1.5 in]{geometry}

\allowdisplaybreaks

\theoremstyle{definition}

\newtheorem*{theorem*}{Theorem}
\newtheorem*{definition*}{Definition}

\newtheorem{theorem}{Theorem}[section]

\newtheorem{lemma}[theorem]{Lemma}
\newtheorem{proposition}[theorem]{Proposition}
\newtheorem{cor}[theorem]{Corollary}

\newcommand{\zz}{\mathbb{Z}}

\newcommand{\cc}{\mathbb{C}}

\newcommand{\ff}{\mathbb{F}}
\newcommand{\bbb}{\mathcal{B}}
\newcommand{\ccc}{\mathcal{C}}

\newcommand{\ttt}{\mathcal{T}}

\newcommand{\ad}{\mathrm{ad}}
\newcommand{\rank}{\mathrm{rank}}
\newcommand{\spann}{\mathrm{span}}
\newcommand{\charr}{\mathrm{char}}
\newcommand{\diag}{\mathrm{diag}}
\newcommand{\tr}{\mathrm{tr}}
\newcommand{\im}{\mathrm{im}}

\newcommand{\quo}[2]{{\raisebox{.2em}{$#1$}\left/\raisebox{-.2em}{$#2$}\right.}}

\begin{document}

\title[On Solvable Lie Algebras of Small Breadth]{On Solvable Lie Algebras of Small Breadth} 

\author[]{Borworn Khuhirun}
\address{Borworn Khuhirun\\ Department of Mathematics and Statistics, Faculty of Science and Technology, Thammasat University, Rangsit Campus, Pathum Thani 12121, Thailand} 
\email{\tt borwornk@mathstat.sci.tu.ac.th}

\author[]{Korkeat Korkeathikhun$^*$}
\thanks{*Corresponding Author}
\address{Korkeat Korkeathikhun\\ Department of Mathematics and Computer Science, Faculty of Science, Chulalongkorn University, Bangkok 10330, Thailand} 
\email{\tt korkeat.k@chula.ac.th, korkeat.k@gmail.com}

\author[]{Songpon Sriwongsa}
\address{Songpon Sriwongsa \\ Department of Mathematics \\ Faculty of Science \\ King Mongkut's University of Technology Thonburi (KMUTT) \\ 126 Pracha-Uthit Road \\ Bang Mod, Thung Khru \\ Bangkok 10140, Thailand}
\email {\tt songpon.sri@kmutt.ac.th}

\author[]{Keng Wiboonton}

\address{Keng Wiboonton \\ Department of Mathematics and Computer Scienc \\ Faculty of Science \\ Chulalongkorn University \\ Bangkok 10330, Thailand}
\email{\tt keng.w@chula.ac.th}

%please add your address here



\begin{abstract}
The concept of breadth has been used in the classification of $p$-groups and nilpotent Lie algebras. In this paper, we investigate this notion for finite-dimensional solvable Lie algebras. Our main focus is to characterize solvable Lie algebras of breadth less than or equal to 2. More importantly, we provide a complete classification of such Lie algebras that are pure and nonnilpotent over the complex numbers.
\end{abstract}

\date{}

\maketitle

%%%%%%%%%%%%%%%%%%%%%%%%%%%%%%%%

\section{Introduction}
\vspace{\baselineskip}
The problem of classifying algebraic objects is a central topic in algebra. In Lie theory, it is well known that nilpotent and solvable Lie algebras have not yet been completely classified; only certain classes have been studied. For example, nilpotent Lie algebras over the complex numbers of dimension at most $7$ have been classified (see \cite{GK96}). The problem becomes increasingly difficult as the dimension grows. To make the classification more manageable, researchers have approached the problem by imposing additional conditions on these Lie algebras. One such condition involves the notion of the {\it breadth} of a Lie algebra. 

In 2015, Khuhirun et al. characterized finite-dimensional nilpotent Lie algebras of breadth $1$ and $2$ \cite{KMS15}. Moreover, the authors used the results to classify finite-dimensional nilpotent Lie algebras of breadth $1$ for any dimension and of breadth $2$ for dimensions $5$ and $6$. Later, Remm extended the work to finite-dimensional nilpotent Lie algebras of breadth $2$ and $3$ via their characteristic sequences \cite{R17}. Independently, in 2021, Sriwongsa et al. gave a characterization of finite-dimensional nilpotent Lie algebra of breadth $3$ over finite fields of odd characteristic \cite{SWK21}. Further developments in this direction, particularly related to the breadth type $(0,3)$, can be found in \cite{KNS231, KNS23}. The notion of breadth for nilpotent Lie algebras is defined analogously to that introduced for finite $p$-groups in \cite{PS99}. That paper focused on the characterization of finite $p$-groups of breadth $1, 2$ and $3$. It should be noted that the results in \cite{PS99} significantly motivated the subsequent work \cite{KMS15, SWK21}.

Given the importance of such classifications, it is also natural to study the problem in the context of solvable Lie algebras. In general, classifying finite-dimensional Lie algebras is challenging due to the vast number of possibilities. The classification of solvable Lie algebras over an arbitrary field has been completed up to dimension $4$ \cite{G05}. Over the real field, solvable Lie algebras have been classified up to dimension $6$ \cite{M63}. For related works on solvable Lie algebras with certain conditions, we refer to \cite{SK10, LCDNV22, T11}.

Motivated by the above discussion, in this paper, we provide a characterization of finite-dimensional solvable Lie algebras over an odd characteristic field of breadth $1$ and over the complex field of breadth $2$ as presented in Section \ref{Basic prop}. We also give a complete classification of finite-dimensional pure solvable nonnilpotent Lie algebras over the complex numbers of breadth $2$. It is worth noting that every nilpotent Lie algebra is solvable, and finite-dimensional nilpotent Lie algebras of breadth $2$ have already been studied in the aforementioned papers. Therefore, we exclude this class of Lie algebras from our consideration. Moreover, to simplify the problem, we assume that the Lie algebras are pure, thereby disregarding direct summands that are abelian ideals.

%%%%%%%%%%%%%%%%%%%%%%%%%%%%%%%%

\section{Basic properties of breadth for Lie algebras} \label{Basic prop}
\vspace{\baselineskip}

Throughout this paper, all Lie algebras are assumed to be finite-dimensional. In this section, we assume that $L$ is a Lie algebra over a field $\ff$ with $\charr(\ff)\neq2$ until Corollary \ref{2.7}. For any $x\in L$ and an ideal $A$ of $L$, \textsl{breadth} of $x$ in $A$ is $b_A(x):=\rank(\ad_x|_A)$ and $b_A(L):=\max\{b_A(x)\mid x\in L\}$. For convenience, we denote $b(x)=b_L(x)$. The \textsl{breadth} of $L$ is $b(L)=b_L(L)=\max\{b(x)\mid x\in L\}$. It is obvious that $b_A(L)\leq\dim[A,L]$ and $b_A(L)\leq b(L)$. Furthermore, $L$ is abelian if and only if $b(L)=0$. Let $T_A=\{x\in L\mid b_A(x)=1\}$. The characterization and classification of Lie algebras of breadth $1$ are known as follows.

%%%%%%%%%%%%%%%%%%%%%%%%%%%%%%%%%%%%%%%%%%%%%%

\begin{theorem}\cite{KMS15}\label{breadthone}
$b(L)=1$ if and only if $\dim[L,L]=1$.
\end{theorem}

%%%%%%%%%%%%%%%%%%%%%%%%%%%%%%%%%%%%%%%%%%%%%%

\begin{proposition}\cite{KMS15}\label{breadthoneclassify}
Let $L$ be an $n$-dimensional nilpotent Lie algebra of breadth 1. Then $L$ has a basis $\{x_1,y_1,x_2,y_2,\ldots,x_k,y_k,z_1,z_2,\ldots,z_{n-2k}\}$ with the nonzero brackets given by $[x_i,y_j]=\delta_{ij}z_1$ and $z_1,z_2,\ldots,z_{n-2k}\in Z(L)$.
\end{proposition}

%%%%%%%%%%%%%%%%%%%%%%%%%%%%%%%%%%%%%%%%%%%%%%

As every ideal of a nilpotent Lie algebra intersects its center nontrivially, we have the following result.

%%%%%%%%%%%%%%%%%%%%%%%%%%%%%%%%%%%%%%%%%%%%%%

\begin{proposition}\label{breadthonenilpotent}
Let $L$ be a Lie algebra of breadth 1. Then $L$ is nilpotent if and only if $[L,L]\subseteq Z(L)$.
\end{proposition}

\begin{proof}
Since $b(L)=1$, we have $\dim[L,L]=1$ by Theorem \ref{breadthone}. If $[L,L]\not\subseteq Z(L)$, then $[L,L]\cap Z(L)=\{0\}$, so $L$ is not nilpotent. The converse is obvious.
\end{proof}

%%%%%%%%%%%%%%%%%%%%%%%%%%%%%%%%%%%%%%%%%%%%%%

\begin{proposition}
If $b(L)\leq1$, then $L$ is solvable.
\end{proposition}

\begin{proof}
If $b(L)=0$, then $L$ is abelian. On the other hand, if $b(L)=1$, then $\dim[L,L]=1$ by Theorem \ref{breadthone}. Thus $[L,L]$ is abelian which is solvable. Hence $L$ is solvable.
\end{proof}

%%%%%%%%%%%%%%%%%%%%%%%%%%%%%%%%%%%%%%%%%%%%%%

For semisimple Lie algebra, we have the following proposition.

%%%%%%%%%%%%%%%%%%%%%%%%%%%%%%%%%%%%%%%%%%%%%%

\begin{proposition}
If $L$ is semisimple, then $b(L)\geq2$.
\end{proposition}

\begin{proof}
Because $L$ is semisimple, it is nonabelian. Therefore $\dim L\geq2$ and $b(L)\neq0$. If $b(L)=1$, then $\dim[L,L]=1$. Thus $L$ contains a proper ideal $[L,L]$, a contradiction. Consequently, $b(L)\geq2$.
\end{proof}

%%%%%%%%%%%%%%%%%%%%%%%%%%%%%%%%%%%%%%%%%%%%%%

A Lie algebra $L$ is called \textsl{pure} or \textsl{stem} if it does not have an abelian ideal as a direct summand. It follows directly that $L$ is pure if and only if $Z(L)\subseteq[L,L]$. In addition, if $L=L_1\oplus L_2$ as a direct sum of Lie algebras, then $b(L)=b(L_1)+b(L_2)$ (c.f. \cite{KMS15}). It is natural to consider pure Lie algebras since the abelian component does not affect the breadth. In the next theorem, we characterize finite-dimensional pure nonnilpotent solvable Lie algebras of breadth 1.

%%%%%%%%%%%%%%%%%%%%%%%%%%%%%%%%%%%%%%%%%%%%%%

\begin{theorem}\label{breadthonesolvable}
Let $n\in\zz_{\geq2}$ and $L$ be an $n$-dimensional nonnilpotent solvable Lie algebra of breadth 1. Then $L$ has a basis $\{x,y,z_1,z_2,\ldots,z_{n-2}\}$ such that $[x,y]=x$ and $z_1,z_2,\ldots,z_{n-2}\in Z(L)$.
\end{theorem}

\begin{proof}
By Theorem \ref{breadthone}, $\dim[L,L]=1$, says $[L,L]=\spann\{x\}$ for some $x\in L$. Because $L$ is nonnilpotent, $[L,L]\nsubseteq Z(L)$ by Proposition \ref{breadthonenilpotent}. Thus $[L,L]\cap Z(L)=\{0\}$ and $x\neq Z(L)$. There exists $y\in L$ such that $[x,y]=x$. Let $V$ be a complementary subspace of $\spann\{x,y\}$ in $L$. For any $u,v\in L$, there exists $\alpha_{uv}\in\ff$ such that $[u,v]=\alpha_{uv}x$. Define a bilinear form $\varphi:V\times V\to\ff$ by $\varphi(u,v)=\alpha_{uv}$ for any $u,v\in L$. Then $\varphi$ is alternating, so there exists a basis $\bbb=\{x_1,y_1,\ldots,x_k,y_k,z_1,\ldots,z_{n-2k-2}\}$ of $V$ such that the matrix of $\varphi$ with respect to $\bbb$ is a block diagonal matrix $\diag(S_1,S_2,\ldots,S_k,0,0,\ldots,0)$ where $S_i=\begin{pmatrix}0&1\\-1&0\end{pmatrix}$ for all $i=1,2,\ldots,k$. As a result, we have $L=\spann\{x,y,x_1,y_1,\ldots,x_k,y_k,z_1,\ldots,z_{n-2k-2}\}$ such that $[x,y]=x, [x_i,y_i]=\delta_{ij}x$, where $\delta_{ij}$ is the Kronecker delta and $z_1,\ldots,z_{n-2k-2}\in Z(L)$. Next, we will claim that $k=0$. Assume that $k\geq1$. Then for any $j\in\{1,2,\ldots,k\}$, we get
$-x=[y,x]=[y,[x_j,y_j]]=[[y,x_j],y_j]+[x_j,[y,y_j]]=0$, a contradiction. Hence $L$ has a basis $\{x,y,z_1,z_2,\ldots,z_{n-2}\}$ such that $[x,y]=x$ and $z_1,z_2,\ldots,z_{n-2}\in Z(L)$.
\end{proof}

%%%%%%%%%%%%%%%%%%%%%%%%%%%%%%%%%%%%%%%%%%%%%%

Moreover, we have the following corollary when $L$ is pure.

%%%%%%%%%%%%%%%%%%%%%%%%%%%%%%%%%%%%%%%%%%%%%%

\begin{cor}\label{2.7}
Let $L$ be a finite-dimensional pure nonnilpotent solvable Lie algebra of breadth 1. Then $L$ has a basis $\{x,y\}$ such that $[x,y]=x$.
\end{cor}

%%%%%%%%%%%%%%%%%%%%%%%%%%%%%%%%%%%%%%%%%%%%%%
%%%%%%%%%%%%%%%%%%%%%%%%%%%%%%%%%%%%%%%%%%%%%%
%%%%%%%%%%%%%%%%%%%%%%%%%%%%%%%%%%%%%%%%%%%%%%

We recall the standard notations of derived series and lower central series from \cite{EW06, H72}. Let $L^{(0)}=L=L^0$ and  $L^{(1)}=[L,L]=L^1$. Define 
\[
L^{(n)}=[L^{(n-1)},L^{(n-1)}] \text{~~~~and~~~~~} L^n=[L,L^{n-1}]
\]
 for all $n\in\zz_{\geq2}$. Now, we focus on characterization of finite-dimensional solvable Lie algebras $L$ of breadth 2. As $L$ is nonabelian, we let $k\in\zz_{\geq2}$ be the smallest integer such that $L^{(k)}=\{0\}$ while $L^{(k-1)}\neq\{0\}$. As a result, $L^{(k-1)}$ is a proper abelian ideal of $L$. The next lemma provides a key step in transition from the class of nilpotent Lie algebras to the class of solvable Lie algebras.

\medskip

The remaining statements in this section are over the complex field $\mathbb{C}$.
%%%%%%%%%%%%%%%%%%%%%%%%%%%%%%%%%%%%%%%%%%%%%%

\begin{lemma}\label{centralizer}
Let $L$ be a nonabelian solvable complex Lie algebra and $A$ be a maximal abelian ideal of $L$. Then $C_L(A)=A$.
\end{lemma}

\begin{proof}
Since $A$ is abelian, $A\subseteq C_L(A)$. Conversely, we suppose that $A\subsetneq C_L(A)$. Since $A$ is an ideal, $C_L(A)$ is also an ideal of $L$ and $\quo{C_L(A)}{A}\neq\{A\}$. Consider an adjoint representation of $L$ on $\quo{C_L(A)}{A}$ given by $x\cdot(y+A)=[x,y]+A$ where $x\in L$ and $y+A\in\quo{C_L(A)}{A}$. It is well-defined since $A$ is an ideal. As $L$ is solvable, $\quo{C_L(A)}{A}$ has a common eigenvector $z+A\neq A$ by Lie's theorem. Then $x\cdot(z+A)=\alpha_x(z+A)$ where $\alpha_x\in\ff$. Thus $[x,z]-\alpha_xz\in A$, so $[x,z]\in A\oplus\spann\{z\}$. Observe that $z\in C_L(A)$ and $z\notin A$, so $A\oplus\spann\{z\}$ is an abelian ideal containing $A$. If $A$ has codimension 1 in $L$, then $L=A\oplus\spann\{z\}$ is abelian, a contradiction. On the other hand, if $A$ has codimension greater than 1 in $L$, then $A\oplus\spann\{z\}$ is a proper abelian ideal containing $A$ which contradicts the maximality of $A$. Consequently, $C_L(A)=A$.
\end{proof}

%%%%%%%%%%%%%%%%%%%%%%%%%%%%%%%%%%%%%%%%%%%%%%

By applying Lemma \ref{centralizer}, the following results and their proofs can be directly generalized from the class of nilpotent Lie algebras in \cite{KMS15} to the class of solvable complex Lie algebras in this paper. In fact, the following Lemmas \ref{NLSB3.3}--\ref{NLSB3.6} correspond to \cite[Lemmas 3.3--3.6]{KMS15}, respectively.

%%%%%%%%%%%%%%%%%%%%%%%%%%%%%%%%%%%%%%%%%%%%%%

\begin{lemma}\label{NLSB3.3}%[mod from NLSB Lemma 3.3]
Let $L$ be a solvable complex Lie algebra of breadth 2, $A$ be a maximal abelian ideal of $L$ and $b_A(L)=2$. Let $x, y, z\in L$ such that $y-z\neq A$ and $b_A(x)=2$. Then at least one of the elements $y, z, y+z, x+y, x+z, x+y+z$ is not in $T_A$.
\end{lemma}

%%%%%%%%%%%%%%%%%%%%%%%%%%%%%%%%%%%%%%%%%%%%%%

\begin{lemma}\label{NLSB3.4}\cite[Lemma 3.4]{KMS15}%[NLSB Lemma 3.4]
~Let $L$ be a complex Lie algebra of breadth 2 and $A$ be an abelian ideal of $L$. Suppose that $b_A(L)=2$. Then $\dim[A,L]=2$.
\end{lemma}

%%%%%%%%%%%%%%%%%%%%%%%%%%%%%%%%%%%%%%%%%%%%%%

\begin{lemma}\label{NLSB3.5}%[mod from NLSB Lemma 3.5]
Let $L$ be a solvable complex Lie algebra of breadth 2 and $A$ be a maximal abelian ideal of $L$. Suppose that $b_A(L)=2$ and $[x,L]\subseteq[A,L]$ for all $x\in L$ with $b_A(x)=2$. Then $[L,L]=[A,L]$ and $\dim[L,L]=2$.
\end{lemma}

%%%%%%%%%%%%%%%%%%%%%%%%%%%%%%%%%%%%%%%%%%%%%%

\begin{lemma}\label{NLSB3.6}%[mod from NLSB Lemma 3.6]
Let $L$ be a solvable complex Lie algebra of breadth 2 and $A$ be a maximal abelian ideal of $L$. Suppose $b_A(L)=1$. Then one of the following holds:
\begin{enumerate}
\item $\dim[A,L]=1$ and $b\left(\quo{L}{[A,L]}\right)<2$, or
\item $\dim\left(\quo{A}{Z(L)}\right)=1$ and $\dim\left(\quo{L}{Z(L)}\right)\leq3$.
\end{enumerate}
\end{lemma}

%%%%%%%%%%%%%%%%%%%%%%%%%%%%%%%%%%%%%%%%%%%%%%

Analogous to \cite[Theorem 3.1]{KMS15}, Lemmas \ref{NLSB3.3}–\ref{NLSB3.6} yield a characterization of finite-dimensional solvable complex  Lie algebras of breadth $2$. We omit the proof, as it follows the same approach as that of \cite[Theorem 3.1]{KMS15}.

%%%%%%%%%%%%%%%%%%%%%%%%%%%%%%%%%%%%%%%%%%%%%%

\begin{theorem}\label{breadthtwo}%[mod from NLSB Theorem 3.1]
Let $L$ be a finite-dimensional solvable complex Lie algebra. Then $b(L)=2$ if and only if one of the following conditions holds:
\begin{enumerate}
\item[(S1)] $\dim[L,L]=2$, or
\item[(S2)] $\dim[L,L]=3$ and $\dim\left(\quo{L}{Z(L)}\right)=3$.
\end{enumerate}
\end{theorem}

%%%%%%%%%%%%%%%%%%%%%%%%%%%%%%%%%%%%%%%%%%%%%%

%%%%%%%%%%%%%%%%%%%%%%%%%%%%%%%%

\section{Classification of solvable Lie algebras of breadth}\label{classification}
\vspace{\baselineskip}
For the rest of this paper, we focus on a classification of finite-dimensional pure nonnilpotent solvable Lie algebra over the field of complex numbers $\mathbb{C}$ of breadth 2. We begin this section by classifying the second case of Theorem \ref{breadthtwo}. Recall that for any Lie algebra $L$, $[C_L([L,L]),C_L([L,L])]\subseteq Z(L)$ and any 3-dimensional nilpotent Lie algebra is either abelian or Heisenberg. For a Lie algebra $L$ over $\mathbb{C}$, $L$ is solvable if and only if $[L,L]$ is nilpotent.

%%%%%%%%%%%%%%%%%%%%%%%%%%%%%%%%%%%%%%%%%%%%%%

\subsection{Solvable Lie algebra of type (S2)}

%%%%%%%%%%%%%%%%%%%%%%%%%%%%%%%%%%%%%%%%%%%%%%

\begin{lemma}\label{lemmaheisenberg}
Let $L$ be a pure nonnilpotent solvable Lie algebra of breadth 2 such that $\dim[L,L]=3$ and $\dim\left(\quo{L}{Z(L)}\right)=3$. Then $[L,L]$ is Heisenberg.
\end{lemma}

\begin{proof}
Assume that $[L,L]$ is not Heisenberg. Since $L$ is solvable, $[L,L]$ is a 3-dimensional nilpotent Lie algebra, so $[L,L]$ is abelian. Because $L$ is pure and nonnilpotent, we have $\dim Z(L)=0, 1$ or 2. If $\dim Z(L)=0$, then $\dim L=3$ and $L=[L,L]$. Therefore $L$ is not solvable, a contradiction. If $\dim Z(L)=1$, then $\dim L=4$ and $Z(L)=\spann\{z\}$ for some $0\neq z\in L$. Extend $Z(L)$ to a basis $\{x,y,z\}$ of $[L,L]$ and choose $w\in L\setminus[L,L]$ so that $L=\spann\{w,x,y,z\}$. As $[L,L]$ is abelian, we get $[x,y]=0$. Thus $[L,L]=\spann\{[w,x],[w,y]\}$, which contradicts the dimension of $[L,L]$. As a result, we have $\dim Z(L)=2$. Let $z_1,z_2\in L$ be such that $Z(L)=\spann\{z_1,z_2\}$. Then we extend $Z(L)$ to $[L,L]=\spann\{z_1,z_2,w\}$.\\
\indent Next, we will claim that there exists $u\in L\setminus[L,L]$ such that $[w,u]\notin Z(L)$. Suppose that $[w,u]\in Z(L)$ for all $u\in L\setminus[L,L]$. Extend $[L,L]$ to $L=\spann\{z_1,z_2,w,x,y\}$. Thus $L^2=[L,[L,L]]=\spann\{[w,x],[w,y]\}\subseteq Z(L)$ which implies $L^3=\{0\}$, contradicts nonnilpotency of $L$. By claim, we can extend $[L,L]=\spann\{z_1,z_2,w\}$ to $L=\spann\{z_1,z_2,w,x,y\}$ so that $[w,x]\notin Z(L)$. Scaling $x$ if necessary, we can additionally assume that
$$
[w,x]=\alpha_1z_1+\beta_1z_2+w,~[w,y]=\alpha_2z_1+\beta_2z_2+\gamma_2w \text{ ~~and~~ } [x,y]=\alpha_3z_1+\beta_3z_2+\gamma_3w
$$
where $\alpha_i,\beta_i,\gamma_i\in\cc$. Replacing $y$ by $y-\gamma_2x$, we may assume that $[w,y]=\alpha_2z_1+\beta_2z_2$. Thus $
0=[w,[x,y]]=[[w,x],y]+[x,[w,y]]=[w,y]$, which implies $[L,L]=\spann\{[w,x],[x,y]\}$, a contradiction.
\end{proof}

%%%%%%%%%%%%%%%%%%%%%%%%%%%%%%%%%%%%%%%%%%%%%%

Now, we can determine the desired Lie algebras of type (S2).

%%%%%%%%%%%%%%%%%%%%%%%%%%%%%%%%%%%%%%%%%%%

\begin{theorem}\label{L23}
Let $L$ be a finite-dimensional pure nonnilpotent solvable Lie algebra of breadth 2 such that $\dim[L,L]=3$ and $\dim\left(\quo{L}{Z(L)}\right)=3$. Then $L=\spann\{x_1,x_2,x_3,z\}$ with nonzero brackets given by $[x_1,x_2]=x_2, [x_1,x_3]=-x_3$ and $[x_2,x_3]=z$.
\end{theorem}

\begin{proof}
By Lemma \ref{lemmaheisenberg}, $[L,L]$ is Heisenberg, so $[L,L]=\spann\{x,y,z\}$ such that $[x,y]=z$ and $[x,z]=0=[y,z]$. Since $L$ is pure, $Z(L)\subseteq[L,L]$, so $Z(L)\subseteq\spann\{z\}$. As before, $\dim Z(L)=0$ is impossible. Hence  $Z(L)=\spann\{z\}=[[L,L],[L,L]]$.\\
\indent Next, we extend $[L,L]=\spann\{x,y,z\}$ to a basis $\{w,x,y,z\}$ of $L$. Then $[w,x]=\alpha_1x+\alpha_2y+\alpha_3z$ and $[w,y]=\beta_1x+\beta_2y+\beta_3z$ where $\alpha_i,\beta_i\in\cc$. Replacing $w$ by $w-\beta_3x+\alpha_3y$, we may assume that $[w,x]=\alpha_1x+\alpha_2y$ and $[w,y]=\beta_1x+\beta_2y$. Let $\bbb=\{x,y\}$ be an ordered basis. Then $V=\spann\:\bbb$ is an $ad_w$-invariant subspace. Observe that $0=[w,z]=[w,[x,y]]=[[w,x],y]+[x,[w,y]]=(\alpha_1+\beta_2)z$, so $\beta_2=-\alpha_1$. Thus $[w,y]=\beta_1x-\alpha_1y$ and $A:=(ad_w\lvert_V)_\bbb=\begin{pmatrix}\alpha_1&\beta_1\\\alpha_2&-\alpha_1\end{pmatrix}$. The eigenvalues of $A$ are $\lambda_1=\sqrt{\alpha_1^2+\alpha_2\beta_1}$ and $\lambda_2=-\lambda_1$. Assume that $\lambda_1=\lambda_2=0$. If $A$ is diagonalizable, then $ad_w\lvert_V=0$. Thus $w\in Z(L)$, a contradiction. As a result, there exists an invertible matrix $P$ such that $PAP^{-1}=\begin{pmatrix}0&1\\0&0\end{pmatrix}$. Thus there exists $0\neq v\in V$ such that $ad_w(v)=0$, so $\spann\{w,v,z\}\subseteq\ker\ad_w$ and $b(w)\leq1$. Thus $[L,L]=\spann\{[w,x],[w,y],[x,y]\}$ has dimension at most two, which is a contradiction.\\
\indent Suppose that $\lambda_1\neq\lambda_2$. Then $A$ is diagonalizable, so there exists an invertible matrix $Q$ such that $QAQ^{-1}=\begin{pmatrix}\lambda_1&0\\0&-\lambda_1\end{pmatrix}$. There exist another basis $\ccc=\{x',y'\}$ of $V$ such that $[w,x']=\lambda_1 x'$ and $[w,y']=-\lambda_1 y'$. Observe that $[L,L]=V\oplus Z(L)$ and $[x',y']\in Z(L)$, so $[x',y']=\gamma z$ for some $\gamma\neq0$. Scaling $z$ if necessary, we can assume that $[x',y']=z$. By setting $x_1=\lambda_1^{-1}w$, $x_2=x'$ and $x_3=y'$, we obtain $L=\spann\{x_1,x_2,x_3,z\}$ with nonzero brackets given by $[x_1,x_2]=x_2, [x_1,x_3]=-x_3$ and $[x_2,x_3]=z$.
\end{proof}

%%%%%%%%%%%%%%%%%%%%%%%%%%%%%%%%%%%%%%%%%%%%%%

\subsection{Solvable Lie algebra of type (S1)}\text{}

Next, we focus on the first case of Theorem \ref{breadthtwo}, which is $\dim[L,L]=2$. Since $L$ is nonnilpotent, we have $\dim L^k\neq0$ for all $k\in\zz_{\geq2}$. Observe that if $\dim[L,[L,L]]=2$, then $[L,[L,L]]=[L,L]$, so dimension of $L^k$ eventually stabilizes at two. Consequently, there are two possibilities, either $\dim L^k=1$ for all $k\in\zz_{\geq2}$ or $\dim L^k=2$ for all $k\in\zz_{\geq2}$.

%%%%%%%%%%%%%%%%%%%%%%%%%%%%%%%%%%%%%%%%%%%%%%

\subsubsection{$\dim L^k=1$ for all $k\in\zz_{\geq2}$}

\begin{lemma}\label{caseb}%[Case B]
Let $L$ be a pure nonnilpotent solvable Lie algebra of breadth 2 such that $\dim[L,L]=2$ and $\dim L^k=1$ for all $k\in\zz_{\geq2}$. Then $\dim Z(L)=1$.
\end{lemma}

\begin{proof}
Since $L$ is pure and nonnilpotent, $\dim Z(L)=0$ or $1$. Suppose that $Z(L)=\{0\}$. Let $0\neq x\in [L,[L,L]]$. Then $[L,[L,L]]=\spann\{x\}$. Extend $[L,[L,L]]$ to $[L,L]=\spann\{x,y\}$. Since $Z(L)=\{0\}$, we have $b(x),b(y)\neq0$. As $[x,L],[y,L]\subseteq[[L,L],L]=\spann\{x\}$, we get $b(x)=b(y)=1$, so $\ker\ad_x$ and $\ker\ad_y$ have codimension 1 in $L$. Note that $C_L([L,L])=\ker\ad_x\cap\ker\ad_y$ and $[C_L([L,L]),C_L([L,L])]\subseteq Z(L)=\{0\}$.\\
\indent Suppose that $[x,y]\neq0$. Then $x\notin\ker\ad_y$ and $y\notin\ker\ad_x$, so $L=C_L([L,L])\oplus\spann\{x,y\}$. Let $u\in C_L([L,L])$ and $z\in L$. Then $z$ can be written as $z=v+ax+by$ where $v\in C_L([L,L])$ and $\alpha,\beta\in\cc$. Thus $[u,z]=[u,v+\alpha x+\beta y]=0$. Therefore $C_L([L,L])\subseteq Z(L)=\{0\}$, so $C_L([L,L])=\{0\}$. As a result, $L=\spann\{x,y\}$ and $\dim[L,L]=1$, a contradiction.\\
\indent Assume that $[x,y]=0$ and $\ker\ad_x\neq\ker\ad_y$. Then $x,y\in C_L([L,L])$ and $\dim C_L([L,L])=\dim L-2$, so there exist $u,v\in L$ such that $u\in\ker\ad_x\setminus C_L([L,L])$ and $v\in\ker\ad_y\setminus C_L([L,L])$.  Thus $L=C_L([L,L])\oplus\spann\{u,v\}$, $[x,u]=0$ and $[y,v]=0$. As $0\neq[x,v],[y,u]\in[[L,L],L]=\spann\{x\}$, we may assume that $[x,v]=x$ and $[y,u]=x$. Since $y\in C_L([L,L])$, we have $0=[y,[u,v]]=[[y,u],v]+[u,[y,v]]=[x,v]=x$, a contradiction.\\
\indent Now, suppose that $[x,y]=0$ and  $\ker\ad_x=\ker\ad_y=C_L([L,L])$. Extend $C_L([L,L])$ to $L=C_L([L,L])\oplus\spann\{u\}$. As $0\neq[x,u],[y,u]\in[[L,L],L]=\spann\{x\}$, we have $[x,u]=\alpha x$ and $[y,u]=\beta x$ where $\alpha,\beta\in\cc\setminus\{0\}$. Replacing $y$ by $\beta x-\alpha y$, we may assume that $[y,u]=0$. Let $z\in L$. Then $z$ can be written as $v+\gamma u$ where $v\in C_L([L,L])$ and $\gamma\in\cc$. Thus $[y,z]=[y,v+\gamma u]=0$, so $y\in Z(L)=\{0\}$, a contradiction.\\
\indent Hence $\dim Z(L)=1$ and that completes the proof.
\end{proof}

%%%%%%%%%%%%%%%%%%%%%%%%%%%%%%%%%%%%%%%%%%%%%%
Here, we present the second type of the Lie algebras in the next theorem.
%%%%%%%%%%%%%%%%%%%%%%%%%%%%%%%%%%%%%%%%%%%%%%%%%%%%
\begin{theorem}\label{L22Lk1}%[Case A]
Let $L$ be a finite-dimensional pure nonnilpotent solvable Lie algebra of breadth 2 such that $\dim[L,L]=2$ and $\dim L^k=1$ for all $k\in\zz_{\geq2}$. Then
\begin{enumerate}
\item\label{even}  $n$ even: $L=\spann\{x_1,x_2,z_1,z_2,\ldots,z_n,z\}$ with the nonzero brackets given by $[x_1,x_2]=x_1, [z_i,z_{i+1}]=z$ for all $i=1,3,5,\ldots,n-1$, or
\item\label{odd} $n$ odd: $L=\spann\{x_1,x_2,z_1,z_2,\ldots,z_n,z\}$ with the nonzero brackets given by $[x_1,x_2]=x_1, [x_2,z_1]=z, [z_i,z_{i+1}]=z$ for all $i=2,4,6,\ldots,n-1$.
\end{enumerate}
\end{theorem}

\begin{proof}
Let $0\neq x\in L^3$. Then $L^3=\spann\{x\}$. Since $[x,L]\in[L^3,L]=L^4=\spann\{x\}$, $\im\:\ad_x=\spann\{x\}$ and $b(x)=1$. By Lemma \ref{caseb}, $\dim Z(L)=1$. Let $0\neq z\in Z(L)\subseteq[L,L]$. Then $[L,L]=\spann\{z,x\}$, which is abelian, so $[L,L]\subseteq C_L([L,L])=\ker\ad_x$. Because $[x,L^2]=\{0\}$, there exists $0\neq y\in L\setminus[L,L]$ such that $[x,y]=x$. As $b(x)=1$, $\ker\ad_x$ has codimension 1 in $L$, so $L=\ker\ad_x\oplus\spann\{y\}$. Additionally, we let $M=C_L(\{x,y\})\subseteq\ker\ad_x=C_L([L,L])$. Observe that $[M,M]\subseteq[C_L([L,L]),C_L([L,L])]\subseteq Z(L)$, so $[M,M]$ is abelian or $[M,M]=Z(L)$. Since $\dim [L,L]=2$, we have $0<b(y)\leq 2$.
\bigskip

\noindent \underline{\bf Case 1:} $b(y)=1$. Then $\ker\ad_y$ has codimension 1 in $L$. Note that $M$ has codimension 2 in $L$ and $L=M\oplus\spann\{x,y\}$. If $[M,M]=\{0\}$, then $M$ is abelian and $[L,L]=\spann\{[x,y]\}=\spann\{x\}$, a contradiction. Therefore $[M,M]=Z(L)$ which is 1-dimensional. As $M^3=\{0\}$, $M$ is a nilpotent Lie algebra of breadth 1 by Theorem \ref{breadthone}. Since $L$ is pure, by Proposition \ref{breadthoneclassify}, $M$ has a basis $\{u_1,v_1,u_2,v_2,\ldots,u_k,v_k,z'\}$ such that $[u_i,v_j]=\delta_{ij}z'$ and $z'\in Z(M)$ for all $i,j\in\{1,2,\ldots,k\}$ for some $k\in\zz_{\geq1}$. Hence $L=M\oplus\spann\{x,y\}=\spann\{x,y,u_1,v_1,u_2,v_2,\ldots,u_k,v_k,z'\}$ with the nonzero brackets given by $[x,y]=x, [u_i,v_j]=\delta_{ij}z'$ and $z'\in Z(L)$ for all $i,j\in\{1,2,\ldots,k\}$ for some $k\in\zz_{\geq1}$. By setting $x_1=x, x_2=y, z_i=u_{\frac{i+1}{2}}, z_{i+1}=v_{\frac{i+1}{2}}$ and $z=z'$ for all $i=1,3,5,\ldots,2k-1$, we obtain the Lie algebra $L$ given in $(\ref{even})$.
\bigskip
%%%%%%%%%%%%%%%%%%%%%

\noindent \underline{\bf Case 2:} $b(y)=2$. Then there exists $w\in L\setminus[L,L]$ such that $[y,w]=z$. Since $\im\;\ad_x=\spann\{x\}$, we have $[x,w]=\alpha x$ for some $\alpha\in\cc$. By replacing $w$ by $w-\alpha y$, we may assume that $[x,w]=0$. Thus $w\in\ker\ad_x=C_L([L,L])$. As $\ker\ad_y$ has codimension 2 in $L$, $M$ has codimension 3 in $L$. Thus $L=M\oplus\spann\{x,y,w\}$ and $C_L([L,L])=M\oplus\spann\{x,w\}$. Observe that $[w,M]\subseteq[C_L([L,L]),C_L([L,L])]\subseteq Z(L)$, so $[w,L]=[w,M\oplus\spann\{x,y,w\}]=Z(L)$. Therefore $b(w)=1$. Since $\dim Z(L)=1$, we have that $[w,M]=\{0\}$ or $[w,M]=Z(L)$.\\
%%%%%%%%%%%%%%%%%%%%%
\indent Case 2.1: $[w,M]=\{0\}$. If $[M,M]=\{0\}$, then $M$ is abelian, so $[M,L]=[M,M\oplus\spann\{x,y,w\}]=\{0\}$. Thus $M\subseteq Z(L)$, which implies $M=Z(L)$. As a result, $L=\spann\{x,y,w,z\}$ with the nonzero brackets given by $[x,y]=x, [y,w]=z$ and $z\in Z(L)$. Now, we set $x_1=x, x_2=y, z_1=w$. Then $L=\spann\{x_1,x_2,z_1,z\}$ with the nonzero brackets given by $[x_1,x_2]=x_1, [x_2,z_1]=z$ as stated in $(\ref{odd})$. On the other hand, if $[M,M]=Z(L)$, then by Theorem \ref{breadthone}, $M$ is a nilpotent Lie algebra of breadth 1 and has a basis $\{u_1,v_1,u_2,v_2,\ldots,u_k,v_k,z',z_1,z_2,\ldots,z_l\}$ with the nonzero brackets given by $[u_i,v_j]=\delta_{ij}z'$ and $z',z_1,z_2,\ldots,z_l\in Z(M)$ for all $i,j\in\{1,2,\ldots,k\}$ for some $k\in\zz_{\geq1}$ and $l\in\zz_{\geq0}$. Observe that $z',z_1,z_2,\ldots,z_l\in Z(L)$ because $[w,M]=\{0\}$. As $\dim Z(L)=1$, we have $l=0$ and there exists $\alpha\in\cc\setminus\{0\}$ such that $[y,w]=z=\alpha z'$. By taking $\spann\{x,y,w\}$ into account and replace $w$ by $\alpha^{-1}w$, we have $L=\spann\{x,y,w,u_1,v_1,u_2,v_2,\ldots,u_k,v_k,z'\}$ with the nonzero brackets given by $[x,y]=x, [y,w]=z', [u_i,v_j]=\delta_{ij}z'$ and $z'\in Z(L)$ for all $i,j\in\{1,2,\ldots,k\}$ for some $k\in\zz_{\geq1}$. By setting $x_1=x, x_2=y, z_1=w, z_i=u_{\frac{i}{2}}, z_{i+1}=v_{\frac{i}{2}}$ and $z=z'$ for all $i=2,4,6,\ldots,2k$, we obtain the Lie algebra $L$ given in $(\ref{odd})$.\\
%%%%%%%%%%%%%%%%%%%%%
\indent Case 2.2:  $[w,M]=Z(L)$. There exists $v\in M$ such that $[w,v]=z$. Observe that $[v,L]=[v,M\oplus\spann\{x,y,w\}]=Z(L)$, so $b(v)=1$. Let $M'=\ker\ad_w\cap\ker\ad_v=C_L(\{w,v\})$ and $N=M\cap M'=C_L(\{x,y,w,v\})$. As $b(w)=b(v)=1$, we have $L=M'\oplus\spann\{w,v\}$. Then $L=N\oplus\spann\{x,y,w,v\}, M=N\oplus\spann\{v\}$ and $[N,N]\subseteq[M,M]\subseteq Z(L)$. If $[N,N]=\{0\}$, then $N$ is abelian, which implies $N=Z(L)$. By replacing $y$ by $y+v$, we have $[x,y]=x$ and $[y,w]=[y,v]=0$. Hence $L=\spann\{x,y,w,v,z\}$ with the nonzero brackets given by $[x,y]=x$, $[w,v]=z$ and $z\in Z(L)$. By setting $x_1=x, x_2=y, z_1=w$ and $z_2=v$, we obtain the Lie algebra $L$ given in $(\ref{even})$. On the other hand, if  $[N,N]=Z(L)$, then by Theorem \ref{breadthone}, $N$ is a nilpotent Lie algebra of breadth 1 and has a basis $\{u_1,v_1,u_2,v_2,\ldots,u_k,v_k,z',z_1,z_2,\ldots,z_l\}$ with the nonzero brackets given by $[u_i,v_j]=\delta_{ij}z'$ and $z',z_1,z_2,\ldots,z_l\in Z(N)$ for all $i,j\in\{1,2,\ldots,k\}$ for some $k\in\zz_{\geq1}$ and $l\in\zz_{\geq0}$. Note that $z',z_1,z_2,\ldots,z_l\in Z(L)$ as $N=C_L(\{x,y,w,v\})$. Since $\dim Z(L)=1$, we have $l=0$ and $Z(L)=\spann\{z\}=\spann\{z'\}$. Without loss of generality, we may additionally assume that $[u_i,v_j]=\delta_{ij}z$ for all $i,j\in\{1,2,\ldots,k\}$ for some $k\in\zz_{\geq1}$. Again, we replace $y$ by $y+v$ so that $[x,y]=x$ and $[y,w]=[y,v]=0$. Hence $L=\spann\{x,y,w,v,u_1,v_1,u_2,v_2,\ldots,u_k,v_k,z\}$ with the nonzero brackets given by $[x,y]=x$, $[w,v]=z$, $[u_i,v_j]=\delta_{ij}z$ for all $i,j\in\{1,2,\ldots,k\}$ for some $k\in\zz_{\geq1}$ and $z\in Z(L)$. Define $x_1=x, x_2=y, z_1=w, z_2= v, z_i=u_{\frac{i-1}{2}}, z_{i+1}=v_{\frac{i-1}{2}}$ and $z=z'$ for all $i=3,5,7,\ldots,2k+1$. Hence we acquire the Lie algebra $L$ given in $(\ref{even})$. This completes the proof.
\end{proof}

%%%%%%%%%%%%%%%%%%%%%%%%%%%%%%%%%%%%%%%%%%%%%%

\subsubsection{$\dim L^k=2$ for all $k\in\zz_{\geq2}$} 
%%%%%%%%%%%%%%%%%%%%%%%%%%%%%%%%%%%%%%%%%%%%%%

\begin{proposition}\label{centerzero}%[Case C extended 1]
Let $L$ be a pure nonnilpotent solvable Lie algebra of breadth 2 such that $\dim[L,L]=2$ and $\dim L^k=2$ for all $k\in\zz_{\geq2}$. Then $Z(L)=\{0\}$ and $C_L([L,L])$ is an abelian ideal of $L$.
\end{proposition}

\begin{proof}
Since $L$ is pure and nonnilpotent, $\dim Z(L)=0$ or $1$. Suppose that $Z(L)=\spann\{z\}$ for some $0\neq z\in L$. Extend $Z(L)$ to $[L,L]=\spann\{x,z\}$. As $[x,L]\subseteq[[L,L],L]=[L,L]$ which is 2-dimensional, we have $b(x)=2$ and $\im\:\ad_x=[L,L]$. There exists $y_1,y_2\in L$ such that $[x,y_1]=x$ and $[x,y_2]=z$. Because $[y_1,y_2]\in[L,L]$, we get $[y_1,y_2]=\alpha x+\beta z$ for some $\alpha,\beta\in\cc$. Therefore
\begin{align*}
0=[x,\alpha x+\beta z]=[x,[y_1,y_2]]=[[x,y_1],y_2]+[y_1,[x,y_2]]=[x,y_2]+[y_1,z]=z,
\end{align*}
which is a contradiction. As a result, $Z(L)=\{0\}$, which implies $C_L([L,L])$ is abelian.
\end{proof}

%%%%%%%%%%%%%%%%%%%%%%%%%%%%%%%%%%%%%%%%%%%%%%

\begin{proposition}\label{central}%[Case C extended 2]
Let $L$ be a pure nonnilpotent solvable Lie algebra of breadth 2 such that $\dim[L,L]=2$ and $\dim L^k=2$ for all $k\in\zz_{\geq2}$. The following are equivalent:
\begin{enumerate}
\item $C_L([L,L])\neq\{0\}$;
\item $[L,L]$ is an abelian ideal of $L$;
\item $[L,L]\subseteq C_L([L,L])$.
\end{enumerate}
\end{proposition}

\begin{proof}
Assume that $C_L([L,L])\neq\{0\}$ and $[L,L]=\spann\{x,y\}$ is nonabelian. Then $[x,y]\neq0$. Next, we will claim that $C_L([L,L])\cap[L,L]=\{0\}$. Suppose the contrary. There exists $0\neq \alpha x+\beta y\in C_L([L,L])$ for some $\alpha,\beta\in\cc$. Therefore $0=[x,\alpha x+\beta y]=\beta[x,y]$ and $0=[y,\alpha x+\beta y]=-\alpha [x,y]$, so $\alpha=\beta=0$, a contradiction. Hence we obtain our claim. Let $V$ be a complementary subspace of $C_L([L,L])$ in $L$ containing $[L,L]$. Then $L=V\oplus C_L([L,L])$ as a vector space and $[V,L]\subseteq[L,L]\subseteq V$. Thus $V$ is an ideal of $L$, so $L=V\oplus C_L([L,L])$ is a Lie algebra direct sum. By Proposition \ref{centerzero}, $C_L([L,L])$ is abelian, so $\{0\}\neq C_L([L,L])\subseteq Z(L)$, contradicts Proposition \ref{centerzero}. The other two implications are straightforward.
\end{proof}

%%%%%%%%%%%%%%%%%%%%%%%%%%%%%%%%%%%%%%%%%%%%%%
\begin{proposition}\label{L22naLk2}%[Case C extended 1]
Let $L$ be a pure nonnilpotent solvable Lie algebra of breadth 2 such that $\dim[L,L]=2$ and $\dim L^k=2$ for all $k\in\zz_{\geq2}$. Then $[L,L]$ is abelian.
\end{proposition}

\begin{proof}
Suppose that $[L,L]$ is nonabelian. Then $[L,L]=\spann\{x,y\}$ such that $[x,y]=x$. If $b(x)=b(y)=1$, then $[L,[L,L]]=\spann\{x\}$, a contradiction. Thus either $b(x)=2$ or $b(y)=2$.\\
\indent Suppose that $b(x)=2$ and $b(y)=1$. There exists $u\in L\setminus[L,L]$ such that $[x,u]=\alpha x+\beta y$ for some $\alpha,\beta\in\cc$ and $\beta\neq0$. As $b(y)=1$, we have $[y,u]=\gamma x$ for some $\gamma\in\cc$. Thus $0=[x,\gamma x ]=[x,[y,u]]=[[x,y],u]+[y,[x,u]]=[x,u]+[y,\alpha x+\beta y]=\beta y$,
a contradiction.\\
\indent Assume that $b(x)=1$ and $b(y)=2$. There exists $u\in L\setminus[L,L]$ such that $[y,u]=\alpha x+\beta y$ for some $\alpha,\beta\in\cc$ and $\beta\neq0$. As $b(x)=1$, we get $[x,u]=\gamma x$ for some $\gamma\in\cc$. Therefore $\beta x=[x,\alpha x+\beta y]=[x,[y,u]]=[[x,y],u]+[y,[x,u]]=[x,u]+[y,\gamma x]=0$, a contradiction.\\
\indent Finally, we suppose that $b(x)=b(y)=2$. Since $b(x)=2$, there exists $u\in L\setminus[L,L]$ such that $[x,u]=\alpha x+\beta y$ where $\alpha,\beta\in\cc$ and $\beta\neq0$. Replacing $u$ by $\beta^{-1}(u-\alpha y)$, we may assume that $[x,u]=y$. Note that $\ker\ad_x\cap\ker\ad_y=C_L([L,L])=\{0\}$ by Proposition \ref{central}. If $u\notin\ker\ad_y$, then $L=\spann\{x,y,u\}$ by counting the codimensions of $\ker\ad_x$ and $\ker\ad_y$. Moreover, we set $[y,u]=\gamma x+\delta y$ where $\gamma,\delta\in\cc$. Therefore $\delta x=[x,\gamma x+\delta y]=[x,[y,u]]=[[x,y],u]+[y,[x,u]]=[x,u]=y$, so $\delta=0$. This is impossible since $b(y)=2$. Hence $u\in\ker\ad_y$. On the other hand, since $b(y)=2$, there exists $v\in L\setminus\spann\{x,y,u\}$ such that $[y,v]=\alpha x+\beta y$ where $\alpha,\beta\in\cc$ and $\beta\neq0$. Replacing $v$ by $\beta^{-1}(v+\alpha x)$, we may assume that $[y,v]=y$. Now, suppose $v\notin\ker\ad_x$ and set $[x,v]=\alpha x+\beta y$ where $\alpha,\beta\in\cc$, not all zero. Then we have
$$
-\alpha x=[y,\alpha x+\beta y]=[y,[x,v]]=[[y,x],v]+[x,[y,v]]=[-x,v]+[x,y]=-\alpha x-\beta y+x,
$$
so $x=\beta y\in\spann\{y\}$, a contradiction. Thus $v\in\ker\ad_x$. In summary, $x,v\in\ker\ad_x$ and $y,u\in\ker\ad_y$, which yields $L=\ker\ad_x\oplus\ker\ad_y=\spann\{x,y,u,v\}$ such that $[x,y]=x$, $[x,u]=y$ and $[y,v]=y$. To inspect $[u,v]$, we let $[u,v]=\gamma x+\delta y$ where $\gamma,\delta\in\cc$. Therefore $y=[y,v]=[[x,u],v]+[u,[x,v]]=[x,[u,v]]=\delta x\in\spann\{x\}$, which yields a contradiction. Hence $[L,L]$ is abelian. This completes the proof.
\end{proof}

%%%%%%%%%%%%%%%%%%%%%%%%%%%%%%%%%%%%%%%%%%%%%%

\begin{proposition}\label{meta}
Let $L$ be a pure nonnilpotent solvable Lie algebra of breadth 2 such that $\dim[L,L]=2$ and $\dim L^k=2$ for all $k\in\zz_{\geq2}$. Then there exists an abelian subalgebra $A$ such that $L$ has a semidirect decomposition $L=A\ltimes[L,L]$.
\end{proposition}

\begin{proof}
By Proposition \ref{centerzero} and Proposition \ref{L22naLk2}, $Z(L)=\{0\}$ and $[L,L]$ is abelian. By \cite[Corollary 3, p. 14]{B75}, $L=H+[L,L]$ as a vector space where $H$ is a Cartan subalgebra of $L$. Suppose that $I:=H\cap[L,L]\neq\{0\}$. Thus $I$ is an ideal of $H$. Since $H$ is nilpotent, $I\cap Z(H)\neq\{0\}$. Let $0\neq x\in I\cap Z(H)$ and $y\in L=H+[L,L]$. Then $y=h+z$ for some $h\in H$ and $z\in[L,L]$. As $x,z\in[L,L]$, we get $[x,z]=0$. Therefore $[x,y]=[x,h+z]=0$, so $x\in Z(L)$, a contradiction. As a result, $I=H\cap[L,L]=\{0\}$ which implies $L=H\ltimes[L,L]$. Observe that $[L,L]=[H,H]+[L,L]$,
so $[H,H]\subseteq H\cap[L,L]=\{0\}$. Hence $H$ is abelian.
\end{proof}

%%%%%%%%%%%%%%%%%%%%%%%%%%%%%%%%%%%%%%%%%%%%%%

Finally, we establish the remaining part of the classification.

%%%%%%%%%%%%%%%%%%%%%%%%%%%%%%%%%%%%%%%%%%%%%%

\begin{theorem}\label{L22aLk2}%[Case C extended 2]
Let $L$ be a finite-dimensional pure nonnilpotent solvable Lie algebra of breadth 2 such that $\dim[L,L]=2$ and $\dim L^k=2$ for all $k\in\zz_{\geq2}$.Then
\begin{enumerate}
\item\label{one} $L=\spann\{x_1,x_2,x_3,x_4,x_5\}$ with nonzero brackets given by $[x_1,x_5]=x_4, ~[x_2,x_4]=x_4$ and $[x_3,x_5]=x_5$, or
\item\label{two} $L=\spann\{x_1,x_2\}\oplus\spann\{x_3,x_4\}$ with nonzero brackets given by $[x_1,x_2]=x_2$ and $[x_3,x_4]=x_4$, or
\item\label{three} $L=\spann\{x_1,x_2,x_3,x_4\}$ with nonzero brackets given by $[x_1,x_4]=x_3, [x_2,x_3]=x_3$ and $[x_2,x_4]=x_4$, or
\item\label{four} $L=\spann\{x_1,x_2,x_3\}$ such that $[x_1,x_2]=x_2$ and $[x_1,x_3]=x_2+x_3$, or
\item\label{five} $L_\gamma=\spann\{x_1,x_2,x_3\}$ such that $[x_1,x_2]=x_2$ and $[x_1,x_3]=\gamma x_3$ where $\gamma\in\cc\setminus\{0\}$.
\end{enumerate}
\end{theorem}

\begin{proof}
By Proposition \ref{meta}, there exists an abelian subalgebra $A$ of $L$ such that $L=A\ltimes[L,L]$. Let $\bbb_A=\{a_1,a_2,\ldots,a_n\}$ be a basis of $A$. As both $A$ and $[L,L]$ are abelian, the structure of $L$ is completely determined by $T_a:=\ad_a\lvert_{[L,L]}$ for all $a\in\bbb_A$. Let $T=\{T_a\mid a\in\bbb_A\}$ and $\ttt=\spann\:T=\{T_a\mid a\in A\}\subseteq \mathfrak{gl}([L,L])$. As $L$ is solvable, so are $A$ and $\ad(A)$. Then $\ttt$ is also solvable. By Lie's Theorem, there exists $0\neq x\in[L,L]$ such that $[a,x]=T_a(x)=\alpha_ax$ where $\alpha_a\in\cc$ for all $a\in A$. Let $\bbb=\{x,y\}$ be a basis of $[L,L]$. Then for any $a\in \bbb_A$, we have $$(T_a)_\bbb
=
\begin{pmatrix}
\alpha_a & \beta_a\\
0 & \gamma_a
\end{pmatrix}$$ 
where $\alpha_a,\beta_a,\gamma_a\in\cc$. If $\beta_a\neq0$, we replace $T_a$ by $T_{\beta_a^{-1}a}$ so that $(T_a)_\bbb=\begin{pmatrix}\alpha'_a&1\\0&\gamma'_a\end{pmatrix}$. Hence we may assume that $(T_a)_\bbb=\begin{pmatrix}\alpha_a&\beta_a\\0&\gamma_a\end{pmatrix}$ where $\beta_a\in\{0,1\}$ and $\alpha_a,\gamma_a\in\cc$ for all $a\in \bbb_A$. Now, we let 
$$
\bbb_{A_0}=\left\{a\in\bbb_A\:\middle|\: (T_a)_\bbb=\begin{pmatrix}\alpha_a&0\\0&\gamma_a\end{pmatrix}\right\}
~~~~\text{ and }~~~~
\bbb_{A_1}=\left\{a\in\bbb_A\:\middle|\: (T_a)_\bbb=\begin{pmatrix}\alpha_a&1\\0&\gamma_a\end{pmatrix}\right\}.
$$
Then $\bbb_A=\bbb_{A_0}\cup\bbb_{A_1}$. Suppose that $|\bbb_{A_1}|\geq2$, says $\bbb_{A_1}=\{a_1,a_2,\ldots,a_m\}$ for some $m\in\zz_{\geq2}$. For every $i\in\{2,3,\ldots,m\}$, we replace $T_{a_i}$ by $T_{a_i-a_1}$ so that $a_i-a_1\in\bbb_{A_0}$. Hence $|\bbb_{A_1}|\in\{0,1\}$.\\
\indent Next, we consider $a\in\bbb_{A_0}$. If $\alpha_a\neq0$, we may replace $T_a$ by $T_{\alpha_a^{-1}a}$ so that $(T_a)_\bbb=\begin{pmatrix}1&0\\0&\gamma'_a\end{pmatrix}$. Therefore, $(T_a)_\bbb=\begin{pmatrix}\alpha_a&0\\0&\gamma_a\end{pmatrix}$ where $\alpha_a\in\{0,1\}$ and $\gamma_a\in\cc$ for all $a\in\bbb_{A_0}$. Let
$$
\bbb_{A_0}^0=\left\{a\in\bbb_{A_0}\:\middle|\: (T_a)_\bbb=\begin{pmatrix}0&0\\0&\gamma_a\end{pmatrix}\right\}
~~~~\text{ and }~~~~
\bbb_{A_0}^1=\left\{a\in\bbb_{A_0}\:\middle|\: (T_a)_\bbb=\begin{pmatrix}1&0\\0&\gamma_a\end{pmatrix}\right\}.
$$
Then $\bbb_{A_0}=\bbb_{A_0}^0\cup\bbb_{A_0}^1$. Similar to the case $\bbb_{A_1}$, we may assume that $|\bbb_{A_0}^1|\in\{0,1\}$. For any $a\in\bbb_{A_0}^0$, if $\gamma_a=0$, then $a\in Z(L)$, which contradicts Proposition \ref{centerzero}. Thus $\gamma_a\neq0$, so we may assume that $(T_a)_\bbb=\begin{pmatrix}0&0\\0&1\end{pmatrix}$ for any $a\in\bbb_{A_0}^0$, which implies $\bbb_{A_0}^0=\left\{a\in\bbb_{A_0}\:\middle|\: (T_a)_\bbb=\begin{pmatrix}0&0\\0&1\end{pmatrix}\right\}$ and $|\bbb_{A_0}^0|\in\{0,1\}$. In summary, $1\leq\dim A\leq3$, so we have three distinct cases to consider.
\bigskip

%%%%%%%%%%%%%%%%%%%%
\noindent \underline{\bf Case 1:} $\dim A=3$. Then $L=\spann\{a_1,a_2,a_3,x,y\}$ such that $[a_1,x]=\alpha x, ~[a_1,y]=x+\beta y, ~[a_2,x]=x, ~[a_2,y]=\gamma y \text{~~and~~} [a_3,y]=y$ where $\alpha,\beta,\gamma\in\cc$.
Define $x_1=a_1-\alpha a_2+(\alpha\gamma-\beta)a_3, x_2=a_2-\gamma a_3, x_3=a_3, x_4=x$ and $x_5=y$. Then $L=\spann\{x_1,x_2,x_3,x_4,x_5\}$ with nonzero brackets given by $[x_1,x_5]=x_4, ~[x_2,x_4]=x_4$ and $[x_3,x_5]=x_5$, which is the Lie algebra given in (\ref{one}).
\bigskip
%%%%%%%%%%%%%%%%%%%%

\noindent \underline{\bf Case 2:} $\dim A=2$. Then we have three subcases.
%%%%%%%%%%%%%%%%%%%%

Case 2.1: $\bbb_{A_1}=\emptyset$. Then we have $L=\spann\{a_2,a_3,x,y\}$ such that $[a_2,x]=x, [a_2,y]=\gamma y$ and $[a_3,y]=y$ where $\gamma\in\cc$. Set $x_1=a_2-\gamma a_3, x_2=x, x_3=a_3$ and $x_4=y$. Then $L=\spann\{x_1,x_2\}\oplus\spann\{x_3,x_4\}$ with nonzero brackets given by $[x_1,x_2]=x_2$ and $[x_3,x_4]=x_4$, which is the Lie algebra given in (\ref{two}).\\
%%%%%%%%%%%%%%%%%%%%
\indent Case 2.2: $\bbb_{A_0}^1=\emptyset$. Then we have $L=\spann\{a_1,a_3,x,y\}$ such that $[a_1,x]=\alpha x, [a_1,y]=x+\beta y$ and $[a_3,y]=y$ where $\alpha,\beta\in\cc$. This case is impossible since $x+\beta y=[a_1,[a_3,y]]=[a_3,[a_1,y]]=\beta y$.\\
%%%%%%%%%%%%%%%%%%%%
\indent Case 2.3: $\bbb_{A_0}^0=\emptyset$. Then $L=\spann\{a_1,a_2,x,y\}$ such that $[a_1,x]=\alpha x, [a_1,y]=x+\beta y, [a_2,x]=x$ and $[a_2,y]=\gamma y$ where $\alpha,\beta,\gamma\in\cc$. Observe that
$$
x+\beta\gamma y=[a_2,x+\beta y]=[a_2,[a_1,y]]=[a_1,[a_2,y]]=[a_1,\gamma y]=\gamma(x+\beta y)=\gamma x+\beta\gamma y,
$$
so $\gamma=1$. Replacing $a_1$ by $a_1-\alpha a_2$, we have $[a_1-\alpha a_2,x]=0 \text{~~and~~} [a_1-\alpha a_2,y]=x+(\beta-\alpha)y=:x+\delta y$ where $\delta=\beta-\alpha\in\cc$. Hence $[a_1,y]=x+\delta y, [a_2,x]=x$ and $[a_2,y]=y$ where $\delta\in\cc$.\\
%%%%%%%%%%%%%%%%%%%%
\indent If $\delta\neq0$, then we notice that
$$
[\delta^{-1}a_1,x]=0, [\delta^{-1}a_1,x+\delta y]=x+\delta y, [a_2-\delta^{-1}a_1,x]=x \text{ and }
[a_2-\delta^{-1}a_1,x+\delta y]=0.
$$
Now, we set $x_1=\delta^{-1}a_1, x_2=x+\delta y, x_3=a_2-\delta^{-1}a_1$ and $x_4=x$. Thus $L=\spann\{x_1,x_2\}\oplus\spann\{x_3,x_4\}$ with nonzero brackets given by $[x_1,x_2]=x_2$ and $[x_3,x_4]=x_4$, which is the Lie algebra given in (\ref{two}).\\
%%%%%%%%%%%%%%%%%%%%
\indent If $\delta=0$, then $L=\spann\{a_1,a_2,x,y\}$. Define $x_1=a_1, x_2=a_2, x_3=x$ and $x_4=y$. Thus $L=\spann\{x_1,x_2,x_3,x_4\}$ with nonzero brackets given by $[x_1,x_4]=x_3, [x_2,x_3]=x_3$ and $[x_2,x_4]=x_4$, which is the Lie algebra given in (\ref{three}). With respect to ordered basis $\{x_1,x_2,x_3,x_4\}$, we get $\ad_{x_1}=E_{34}, \ad_{x_2}=E_{33}+E_{44}, \ad_{x_3}=-E_{32}, \ad_{x_4}=-E_{31}-E_{42}$, so $ad(L)=\spann\{E_{34},E_{33}+E_{44},E_{32},E_{31}+E_{42}\}$ and $C_{g\ell(4,\cc)}(\ad(L))=\spann\{I_4, E_{12}+E_{34}\}$ where $E_{ij}$ is the standard basis matrix with 1 in the $i$th row and $j$th column and 0's elsewhere. By \cite{RWZ88}, $L$ is indecomposable because $\tr((E_{12}+E_{34})^2)=0$. Hence it is not isomorphic to one in the case $\delta\neq0$.
\bigskip
%%%%%%%%%%%%%%%%%%%%

\noindent\underline{\bf Case 3:} $\dim A=1$. Then we also have three subcases to consider.\\
%%%%%%%%%%%%%%%%%%%%
\indent Case 3.1: $\bbb_{A_1}=\bbb_{A_0}^1=\emptyset$. Then $L=\spann\{a_3,x,y\}$ such that $[a_3,y]=y$, so $b(L)=1$, a contradiction.\\
%%%%%%%%%%%%%%%%%%%%
\indent Case 3.2: $\bbb_{A_1}=\bbb_{A_0}^0=\emptyset$. Then $L=\spann\{a_2,x,y\}$ such that $[a_2,x]=x$ and $[a_2,y]=\gamma y$ where $\gamma\in\cc\setminus\{0\}$. Note that $\gamma\neq0$, otherwise $b(L)=1$. By setting $x_1=a_2, x_2=x$ and $x_3=y$, we obtain $L_{1,\gamma}:=\spann\{x_1,x_2,x_3\}$ such that $[x_1,x_2]=x_2$ and $[x_1,x_3]=\gamma x_3$ where $\gamma\in\cc\setminus\{0\}$.\\
%%%%%%%%%%%%%%%%%%%%
\indent Case 3.3: $\bbb_{A_0}^1=\bbb_{A_0}^0=\emptyset$. Then $L=\spann\{a_1,x,y\}$ such that $[a_1,x]=\alpha x$ and $[a_1,y]=x+\beta y$ where $\alpha,\beta\in\cc\setminus\{0\}$. Note that $\alpha\neq0$ and $\beta\neq0$, otherwise $b(L)=1$. Set $\delta=\beta\alpha^{-1}\neq0$. Then
$$
[x,\alpha y]=0, [\alpha^{-1}a_1,x]=x \text{ and }[\alpha^{-1}a_1,\alpha y]=x+\beta y=x+\beta\alpha^{-1}(\alpha y)=x+\delta(\alpha y),
$$
so we define $x_1=\alpha^{-1}a_1, x_2=x$ and $x_3=\alpha y$. Then we have $L_{2,\delta}:=\spann\{x_1,x_2,x_3\}$ such that $[x_1,x_2]=x_2$ and $[x_1,x_3]=x_2+\delta x_3$ where $\delta\in\cc\setminus\{0\}$.

In $L_{2,\delta}$, we notice that $[x_1, x_2+(\delta-1)x_3]=\delta(x_2+(\delta-1)x_3)$. Thus $L_{1,\gamma}$ is isomorphic to $L_{2,\gamma}$ via an isomorphism $\varphi:L_{1,\gamma}\to L_{2,\gamma}$ defined by $\varphi(x_1)=x_1, \varphi(x_2)=x_2 \text{~~~~and~~~~} \varphi(x_3)=x_2+(\gamma-1)x_3$ where $\gamma\in\cc\setminus\{0,1\}$. By including $\gamma=1$ into account, we finally have two possibilities as follows:
\begin{itemize}
\item $L=\spann\{x_1,x_2,x_3\}$ such that $[x_1,x_2]=x_2$ and $[x_1,x_3]=x_2+x_3$, which is the Lie algebra give in (\ref{four}), or
\item $L_\gamma=\spann\{x_1,x_2,x_3\}$ such that $[x_1,x_2]=x_2$ and $[x_1,x_3]=\gamma x_3$ where $\gamma\in\cc\setminus\{0\}$, which is the Lie algebra give in (\ref{five}).
\end{itemize}
In fact, $L_\gamma$ is isomorphic to $L_\delta$ if and only if $\gamma=\delta$ or $\gamma=\delta^{-1}$ as referred in \cite{EW06}, while $L$ is not isomorphic to $L_\gamma$. To see this, we suppose that there exists an isomorphism $\varphi:L\to L_\gamma$ defined by
$$
\varphi(x_1)=ax_1+bx_2+cx_3,~~~~~\varphi(x_2)=dx_2+ex_3 \text{~~~~and~~~~} \varphi(x_3)=fx_2+gx_3
$$
where $a,b,c,d,e,f,g,h\in\cc$. Note that $a\neq0$, otherwise $\varphi$ is not surjective. As $\varphi$ is a bijection, we have $a(dg-ef)=\det(\varphi)_\bbb\neq0$ where $\bbb=\{x_1,x_2,x_3\}$ is an ordered basis for $L_\gamma$. Thus $a\neq0$ and $dg-ef\neq0$. As $\varphi([x_1,x_3])=[\varphi(x_1),\varphi(x_3)]$, we have $af=d+f$ and $ag\gamma=e+g$. Similarly, the condition $\varphi([x_1,x_2])=[\varphi(x_1),\varphi(x_2)]$ yields $ad=d$ and $ae\gamma=e$. The first equality implies $a=1$ or $d=0$, while the other yields $e=0$ or $\gamma=a^{-1}$. Thus we consider four cases. 
\begin{enumerate}
    \item[(i)] If $a=1$ and $e=0$, then $af=d+f$ implies $d=0$ and $dg-ef=0$.
    \item[(ii)] If $a=1$ and $\gamma=a^{-1}=1$, then $af=d+f$ implies $d=0$. Also, $ag\gamma=e+g$ yields $e=0$, so $dg-ef=0$.
    \item[(iii)] If $d=0$ and $e=0$, then clearly $dg-ef=0$.
    \item[(iv)] If $d=0$ and $\gamma=a^{-1}=1$, then $ag\gamma=e+g$ yields $e=0$, so $dg-ef=0$.
\end{enumerate}
Hence all cases eventually lead to $dg-ef=0$, a contradiction. Consequently, $L$ and $L_\gamma$ are not isomorphic.
\end{proof}

%%%%%%%%%%%%%%%%%%%%%%%%%%%%%%%%

\section{Concluding remarks}
\vspace{\baselineskip}
We have completely classified all finite-dimensional pure solvable nonnilpotent Lie algebra of breadth 2 over $\mathbb{C}$ as summarized in the table below.
%%%%%%%%%%%%%%%%%%%%%%%%%%%%%%%%%%%%%%
\begin{center}
\resizebox{14cm}{!}{
\begin{tabular}{|c|L{5.25cm}|L{7cm}|}
\hline
\textbf{Theorem} & \textbf{Basis} & \textbf{Nonzero brackets} \\ 
\hline
\ref{L23} & $L_1=\spann\{x_1,x_2,x_3,z\}$ & $[x_1,x_2]=x_2, [x_1,x_3]=-x_3, [x_2,x_3]=z$\\ 
\hline
\multirow{2}{*}{\ref{L22Lk1}} & $L_2=\spann\{x_1,x_2,z_1,z_2,\ldots,z_n,z\}$ & $[x_1,x_2]=x_1, [z_i,z_{i+1}]=z$ \hspace{70pt} for all $i=1,3,5,\ldots,n-1$ where $n$ is even\\
\cline{2-3}
& $L_3=\spann\{x_1,x_2,z_1,z_2,\ldots,z_n,z\}$ & $[x_1,x_2]=x_1, [x_2,z_1]=z, [z_i,z_{i+1}]=z$ \hspace{20pt} for all $i=2,4,6,\ldots,n-1$ where $n$ is odd\\
\hline
\multirow{5}{*}{\ref{L22aLk2}} & $L_4=\spann\{x_1,x_2,x_3,x_4,x_5\}$ & $[x_1,x_5]=x_4, ~[x_2,x_4]=x_4$, $[x_3,x_5]=x_5$\\
\cline{2-3}
& $L_5=\spann\{x_1,x_2\}\oplus\spann\{x_3,x_4\}$ & $[x_1,x_2]=x_2$, $[x_3,x_4]=x_4$\\
\cline{2-3}
& $L_6=\spann\{x_1,x_2,x_3,x_4\}$ & $[x_1,x_4]=x_3, [x_2,x_3]=x_3$, $[x_2,x_4]=x_4$\\
\cline{2-3}
& $L_7=\spann\{x_1,x_2,x_3\}$ & $[x_1,x_2]=x_2$, $[x_1,x_3]=x_2+x_3$\\
\cline{2-3}
& $L_{8,\gamma}=\spann\{x_1,x_2,x_3\}$ & $[x_1,x_2]=x_2$, $[x_1,x_3]=\gamma x_3$ where $\gamma\in\cc\setminus\{0\}$\\
\hline
\end{tabular}
}
\end{center}
%%%%%%%%%%%%%%%%%%%%%%%%%%%%%%%%%%%%%%
\medskip
Note that we have the isomorphism condition for the last class of Lie algebra: 
\[
L_{8,\gamma}\cong L_{8,\delta} \text{ if and only if } \gamma=\delta \text{ or } \gamma=\delta^{-1}.
\]

%%%%%%%%%%%%%%%%%%%%%%%%%%%%%%%%%%%%%%

In \cite{G05}, solvable Lie algebras of dimension 3 and 4 have been classified. Among these, solvable Lie algebras over $\cc$ of breadth 2 are $L^2, L_a^3 (a\neq0), L_a^4 (a\neq0), M_0^3, M_{0,b}^6 (b\in\cc), M_{0,b}^7 (b\in\cc), M^8, M_0^{13}$ and $M_a^{14} (a\neq0)$. Note that $M_{0,0}^7$ is nilpotent. If $b\neq0$, then $M_{0,b}^7$ is not pure as $bx_1-x_3\in Z(M_{0,b}^7)\setminus[M_{0,b}^7,M_{0,b}^7]$. When $b\neq0$, $M_{0,b}^6$ is not pure because $-bx_1-x_2+x_3\in Z(M_{0,b}^6)\setminus[M_{0,b}^6,M_{0,b}^6]$. Analogously, $M_0^3$ is not pure since $x_2-x_3\in Z(M_0^3)\setminus[M_0^3,M_0^3]$. Note that $M_a^{14}$ are all isomorphic over $\cc$, so only $M_1^{14}$ will be considered. By the same reason, we also consider $L_1^4$. Additionally, $\dim L_2\geq5$ and $\dim L_4=5$, so they are excluded from comparison. Consequently, our results, $L_1, L_3 (n=1), L_5, L_6, L_7, L_{8,\gamma} (\gamma\neq0)$, coincide with $L^2, L_a^3 (a\neq0), L_1^4, M_{0,0}^6, M^8, M_0^{13}$ and $M_1^{14}$ as follows:
\begin{enumerate}
\item $L_1$ is isomorphic to $M_1^{14}$ via $\varphi_1:L_1\to M_1^{14}$ defined by $\varphi_1(x_1)=-x_4, \varphi_1(x_2)=x_1-x_3, \varphi_1(x_3)=x_1+x_3$ and $\varphi_1(z)=-2x_2$.
\item When $n=1$, $L_3$ is isomorphic to $M_{0,0}^6$ via $\varphi_3:L_3\to M_{0,0}^6$ defined by $\varphi_3(x_1)=x_3, \varphi_3(x_2)=-x_4, \varphi_3(z_1)=x_2-x_1$ and $\varphi_3(z)=x_2-x_3$.
\item $L_5$ is isomorphic to $M^8$ via $\varphi_5:L_5\to M^8$ defined by $\varphi_5(x_1)=x_1, \varphi_5(x_2)=x_2, \varphi_5(x_3)=x_3$ and $\varphi_5(x_4)=x_4$.
\item $L_6$ is isomorphic to $M_0^{13}$ via $\varphi_6:L_6\to M_0^{13}$ defined by $\varphi_6(x_1)=x_1-x_3, \varphi_6(x_2)=x_4, \varphi_6(x_3)=-x_2$ and $\varphi_6(x_4)=x_1$. 
\item $L_7$ is isomorphic to $L_{-1/4}^3$ via $\varphi_7:L_7\to L_{-1/4}^3$ defined by $\varphi_7(x_1)=2x_3, \varphi_7(x_2)=x_1-2x_2$ and $\varphi_7(x_3)=-x_1$.
\item $L_{8,1}$ is isomorphic to $L^2$ via $\varphi_8:L_{8,1}\to L^2$ defined by $\varphi_8(x_1)=x_3, \varphi_8(x_2)=x_1$ and $\varphi_8(x_3)=x_2$.
\item $L_{8,-1}$ is isomorphic to $L_1^4$ via $\psi_8:L_{8,-1}\to L_1^4$ defined by $\psi_8(x_1)=x_3, \psi_8(x_2)=x_1+x_2$ and $\psi_8(x_3)=x_1-x_2$.
\item When $\gamma\in\cc\setminus\{0,\pm1\}$ and $a\in\cc\setminus\{0,-1/4\}$, $L_{8,\gamma}$ is isomorphic to $L_a^3$ via $\chi_8:L_{8,\gamma}\to L_a^3$ defined by $\chi_8(x_1)=\lambda_2^{-1}x_3, \chi_8(x_2)=-\lambda_1x_1+x_2$ and $\chi_8(x_3)=-\lambda_2x_1+x_2$, where $\lambda_1$ and $\lambda_2$ are two distinct roots of $x^2-x-a$ such that $\gamma=\lambda_1\lambda_2^{-1}$.
\end{enumerate}

%%%%%%%%%%%%%%%%%%%%%%%%%%%%%%%%

\section{Acknowledgements}
The authors are very grateful to the referee for careful reading and valuable comments. This study is supported by Thammasat University Research Fund, Contract No. TUFT 108/2567. This research project is supported by grants for development of new faculty staff, Ratchadaphiseksomphot Fund, Chulalongkorn University.

%%%%%%%%%%%%%%%%%%%%%%%%%%%%%%%%

%%%%%%%%%%%%%%%%%%%%%%%%%%%%%%%%

%\newpage
%\input{9-Appendix.tex}

\end{document}